\documentclass[12pt, reqno]{amsart}
\usepackage{amsmath}
\usepackage{amsfonts}
\usepackage{amssymb}
\usepackage{amsthm}
\usepackage[shortlabels]{enumitem}
\usepackage{graphicx}
\usepackage{color}
\usepackage{dsfont}
\usepackage{MnSymbol}
\usepackage{enumitem}
\usepackage{extpfeil}
\usepackage{mathtools}
\usepackage{url}
\usepackage[margin=1in]{geometry}
\usepackage{fancyhdr}
\usepackage[all,cmtip]{xy}

\newtheorem{theorem}{Theorem}[section]
\newtheorem{lemma}[theorem]{Lemma}

\newtheorem{proposition}[theorem]{Proposition}
\newtheorem{cor}[theorem]{Corollary}

\theoremstyle{remark}

\theoremstyle{definition}

\def\mapnew#1{\smash{\mathop{\longrightarrow}\limits^{#1}}}

\def\<{\langle}
\def\>{\rangle}
\def\-{\overline}

\numberwithin{equation}{section}

\begin{document}

\title[]{Lie algebras and coherence}

\author{Dessislava H. Kochloukova}
\address
{Department of Mathematics, State University of Campinas (UNICAMP), 13083-859, Campinas, SP, Brazil }
\email{desi@unicamp.br}

\thanks{This work was started during a visit at the Department of Mathematics, UC-Riverside in May 2024. The author thanks Stefano Vidussi for the hospitality,the math discussions  and the local financial support. The author is partially supported by  CNPq 305457/2021-7 and FAPESP 2024/14914-9.}
%    author two information

%    \subjclass is required.

\date{\today}

\dedicatory{}

%    "Communicated by" -- provide editor's name; required.
\commby{}

%    Abstract is required.
\begin{abstract}  We determine some sufficient conditions for the split extension of two free finitely generated non-abelian  Lie algebras   $L = F_1 \leftthreetimes F_2$ over an infinite field $K$ to be incoherent. 
\end{abstract}

\maketitle

%    Text of article.

%    Bibliographies can be prepared with BibTeX using amsplain,
%    amsalpha, or (for "historical" overviews) natbib style.
\bibliographystyle{amsplain}
%    Insert the bibliography data here.
[

\section*{Introduction}

In this paper we study incoherence of Lie algebras. All Lie algebras we consider are over a fixed field $K$.
As defined by Roos, Passman and Small in \cite{P-S} and \cite{Roos} a Lie algebra is coherent if every finitely generated Lie subalgebra is finitely presented (in terms of generators and relations) and  we call a Lie algebra incoherent if it is not coherent. Recall that a Lie algebra $L$ is finitely presented if there is a free finitely generated  Lie algebra $F$ and an epimorphism of Lie algebras $\varphi : F \to L$ such that $R = Ker(\varphi)$ is finitely generated as an ideal of $F$. \iffalse{Note that by a version of the Hopf formula for Lie algebras  the homology group $$H_2(L, K) \simeq (R \cap [F,F])/ [R, F],$$ in particular if $L$ is finitely presented then $dim_K H_2(L,K) < \infty$.}\fi

As shown by Roos and later proved with different argument by Passman and Small if $L = F(a,b) \oplus F(c,d)$ is the direct sum of the free Lie algebras $F(a,b)$ and $F(c,d)$ with free basis $\{ a,b \}$ and $\{ c, d \}$, then the Lie subalgebra $R = \langle a, b+ c, d \rangle$ is an ideal of $L$ of codimension 1 that is not finitely presented as a Lie algebra. This provides an example of a Lie algebra that is incoherent.

In \cite{K-MP1} Kochloukova and Martinez-Perez studied subdirect sums of Lie algebras, generalising  the result of Roos and Passman-Small. In particular they showed that if $S$ is a subdirect sum of $F_1 \oplus \ldots \oplus F_k$ such that each $F_i$ is a finitely generated, free Lie algebra and $S \cap F_i \not= 0$ for every $i$ and $S$ is of homological type $FP_2$ then for each projection map $p_{i,j} :  F_1 \oplus \ldots \oplus F_n \to F_i \oplus F_j$, we have that $p_{i,j} (S) = F_i \oplus F_j$. Note that a Lie algebra $L$ is of homological type $FP_2$ if and only if $L \simeq F/ R$, where $F$ is a free, finitely generated Lie algebra and $R/ [R,R]$ is finitely generated as a $U(L)$-module via the adjoint action. Note that every finitely presented Lie algebra is of type $FP_2$ but whether the converse holds is  an open problem.

There is little known about  finite presentability (in terms of generators and relations) of Lie algebras. Still the case of metabelian Lie algebras and center-by-metabelian Lie algebras is completely understood. In \cite{B-G1}, \cite{B-G2}, \cite{B-G3} using methods from commutative algebra Bryant and Groves developed the classification of finitely presented metabelian and center-by-metabelian Lie algebras.

As it is the case with finite presentability, little is known about
coherence of Lie algebras but the same notion in the category of discrete groups is well studied. There is an open conjecture due to Kropholler, Walsh and independently Wise that if $F_1$ and $F_2$ are finitely generated, non-abelian free groups then the semi-direct product $F_1 \rtimes F_2$ is incoherent. Though the conjecture is still open many cases were settled by Kropholler, Walsch, Vidussi, Kochloukova  in   \cite{K-V}, \cite{K-V-W}, \cite{K-W}. In all cases the results obtained use significantly Bieri-Strebel-Neumann $\Sigma$-invariants that were originally defined in \cite{BNS}.
The case of incoherence in the category of pro-$p$ groups was treated \iffalse{by Kochloukova}\fi in \cite{K}. Moving to another category we note that incoherence of associative algebras was studied by Small and Zelmanov  in \cite{S-Z}, where they showed that the algebra of generic matrices $G(m,n)$ is incoherent.

We believe that a Lie algebra version of the Kropholler-Walsh-Wise Conjecture is a natural conjecture to explore in the category of Lie algebras.  For Lie algebras $A$ and $B$ we write $ A \leftthreetimes B$ for the semidirect sum of the Lie algebras $A$ and $B$.

\medskip

{\bf Conjecture} {\it 
 Let $L = F_1 \leftthreetimes F_2$ be a Lie algebra over an infinite field $K$, where $F_1$ and $F_2$ are finitely generated  non-abelian free Lie algebras. Then $L$ is incoherent. }

\medskip
 A $\mathbb{N}$-graded Lie algebra $L$ has a filtration 
$$L = \oplus_{i \geq 1} L_i, \hbox{  where }[L_i, L_j] \subseteq L_{i+ j} \hbox{ for all }i,j \geq 1.$$ As pointed by Weigel in \cite{We} for a  $\mathbb{N}$-graded Lie algebra $L$ we have:

1) $L$ is finitely generated if and only if $dim_K H_1(L, K) < \infty$;

2) $L$ is finitely presented if and only if  $dim_K H_i(L, K) < \infty$ for $i \leq 2$.

The following theorem is our main result. It answers positively the above conjecture under some additional hypothesis and it can be viewed as a generalization of the Roos and Passman-Small result.

\medskip
{\bf Theorem A}  {\it Let $L = F_1 \leftthreetimes F_2$ be a $\mathbb{N}$-graded  Lie algebra over an infinite field $K$, where 

a) $F_1$ and $F_2$ are finitely generated $\mathbb{N}$-graded  Lie  subalgebras, 

b) $F_2$ is free and non-abelian,

c) $F_1$ has an $\mathbb{N}$-graded ideal $N$ such that $N$ is not finitely generated ( as a Lie algebra), $F_1 = N \oplus K y$ as a vector space, $y$ a homogeneous element,  $[L,L] \cap F_1 \subseteq N$

d)  there is a free basis of $F_2$ that contains a homogeneous element $s_1$ that has the same degree as $y$. 

 Then there is an ideal $L_0$ of $L$ such that $dim_K (L/ L_0) = 1$ and $L_0$ as a Lie algebra is finitely generated but it is not finitely presented. In particular $L$ is incoherent.}

\medskip
The reason we need that $K$ is an infinite field is that we need a very specific form of the Noether normalization lemma from commutative algebra that requires the ground field to be infinite, see Theorem \ref{Noether}.

A natural strategy  to prove Theorem A is to try to transfer the methods used in the case of discrete groups to the case of Lie algebras. Unfortunately some of the main ingredients of the group theoretic results in \cite{K-V-W}, \cite{K-W} do not have counterparts in the category of Lie algebras.
We need to work with $\mathbb{N}$-graded Lie algebras because we use  Proposition \ref{sigma-Lie}, whose proof requires $\mathbb{N}$-graded Lie algebras.  In Lemma \ref{not-redundant} we show that Proposition \ref{sigma-Lie} does not hold for Lie algebras that are not $\mathbb{N}$-graded. 

We call a surface Lie algebra the Lie algebra with a presentation ( in terms of generators and relator) $\langle x_1, \ldots, x_{2n} \ | \ [x_1, x_2] + \ldots + [x_{2n-1} , x_{2n} ] = 0 \rangle$ for some $n \geq 1$. We call $x_1, \ldots, x_{2n}$ a standard set of generators. For a free Lie algebra we call a standard set of generators any free basis. For a Lie algebra $L$ we write $d(L)$ for the minimal number of generators of $L$.

Theorem A implies  the following results.

\medskip
{\bf Corollary B1} {\it Let $L = F_1 \leftthreetimes F_2$ be a Lie algebra over an infinite field $K$, where 

a) $F_2$ is a finitely generated  non-abelian free Lie algebra, 

b) $F_1$ is a non-abelian free or surface Lie algebra such that $F_1 \not\subseteq [L,L]$, 

c) $L$ is $\mathbb{N}$-graded with $F_1$ and $F_2$ $\mathbb{N}$-graded subalgebras of $L$, 

d) $F_1$ has a standard set of homogeneous generators all of the same degree, say $n_0$, and $F_2$ has a  free basis with at least one element that is homogeneous of degree $n_0$. 

Then there is an ideal $L_0$ of $L$ such that $dim_K (L/ L_0) = 1$ and $L_0$ as a Lie algebra is finitely generated but it is not finitely presented. In particular $L$ is incoherent.}

\medskip

{\bf Corollary B2} {\it Let $L = F_1 \leftthreetimes F_2$ be a Lie algebra over an infinite field $K$, where 

a) $F_2$ is a finitely generated  non-abelian free Lie algebra,

b)  $F_1$ is a non-abelian free or surface Lie algebra,

c) $L$ is $\mathbb{N}$-graded with  $F_1$ and $F_2$ $\mathbb{N}$-graded subalgebras of $L$, both generated by  homogeneous elements all of degree one, 

d) $d(F_2) > d(F_1)^2$. 

Then  $L$ is incoherent.}

\medskip

\medskip
In Proposition \ref{coherent0} we show that
if $L = F \leftthreetimes Q$ is a $\mathbb{N}$-graded Lie algebra with $F$ finitely generated free,  Lie subalgebra and $Q$ one dimensional  Lie subalgebra, then $L$ is graded coherent i.e. every finitely generated graded subalgebra is finitely presented. Note that in the above we do not suppose that $F$ and $Q$ are graded subalgebras of $L$.  We conjecture that the same holds without the grading condition. The group theoretic case of the following conjecture was proved in \cite{F-H}.

\medskip
{\bf Conjecture} {\it Let $L = F \leftthreetimes Q$ be a Lie algebra with $F$ finitely generated free Lie subalgebra and $dim_K Q = 1$. Then $L$ is  coherent.}

\medskip
In the proofs presented in this paper whenever possible we give homology free arguments. We separate the results that use homological proofs in Section \ref{sec-ex} and Section \ref{secEuler}.

\section{Preliminaries}  \label{sec-hnn}

\subsection{$\mathbb{N}$-graded Lie algebras}

All Lie algebras we consider are over a fixed field $K$.
For a Lie algebra $L$ denote by $U(L)$ the universal enveloping algebra of $L$. Thus $U(L)$ is an associative $K$-algebra with 1, that contains $L$ as a $K$-subspace and for $a,b \in L$ we have that the element $[a,b] \in L$ equals  $ ab - ba$ in $ U(L)$.

A left (resp. right) $L$-module $V$ is a left (resp. right) $U(L)$-module i.e. $[a,b]v = abv - bav$ (resp. $v[a,b] = vab- vba$) for all $v \in V, a,b \in L$.

Denote by $\mathbb{N}$ the set of positive integers $\{ 1,2,\ldots \}$.
 A $\mathbb{N}$-graded Lie algebra $L$ has a decompositions as a direct sum of vector subspaces $L_i$ i.e.
$$L = \oplus_{i \geq 1} L_i, \hbox{  where }[L_i, L_j] \subseteq L_{i+ j} \hbox{ for all }i,j \geq 1.$$ \iffalse{By \cite{We} for an  $\mathbb{N}$-graded Lie algebra $L$ we have:

1) $L$ is finitely generated if and only if $dim_K H_1(L, K) < \infty$,

2) $L$ is finitely presented if and only if  $dim_K H_i(L, K) < \infty$ for $i \leq 2$.

Furthermore}\fi

\begin{lemma} \cite{K-MP3} Let $L$ be an $\mathbb{N}$-graded Lie algebra with an $\mathbb{N}$-graded presentation $\langle X | R_0 \rangle$,
with $X$ minimal and $R_0$ minimal possible once $X$ is fixed. Then
$|X| = dim_K H_1(L,K), |R_0| = dim_K H_2(L,K)$.
\end{lemma}

For an $\mathbb{N}$-graded Lie algebra $L$, a left  $U(L)$-module  $V$ is called graded if there is a decomposition of vector spaces $$V = \oplus_{i \geq 0} V_i, \hbox{ where } L_i V_j \subseteq V_{i+ j}$$  \iffalse{Note that $U(L)$ is a graded left (resp. right) $U(L)$ with $U(L)_i$ for $i > 0$ defined as the $K$-subspace generated by all $L_{i_1} \ldots L_{i_k}$, where $i_1 + \ldots + i_k = i, k \geq 1$ and $U(L)_0 = K.1 = K$. This is a grading of $U(L)$  simultaneously as a left and right $U(L)$-module.}\fi
For a graded left $U(L)$-module $V$ we have that $V$ is finitely generated if and only if $K \otimes_{U(L)} V$ is finite dimensional, where $K$ is considered as a trivial right $U(L)$-module i.e. $L$ acts as 0.

 \subsection{Some definitions and results on HNN extensions of Lie algebras}  \ \

 1) First
  we recall some definitions and results from \cite{B-K} on Lyndon-Shirshov monomials
  .
 Let $T$ be a set, $T^*$ the set of all associate words on  $T$ and $T^{\#}$ be the set of all non-associate words on $T$. 
 
 Fix a linear order $< $ in the alphabet $T$ and consider the lexicographic order $< $ in $T^*$  i.e. for $u \in T^* \setminus \{ 1 \}$ we have   $u < 1$ and $x_i w < x_j v$ provided $x_i < x_j$ or $x_i = x_j$ and $w < v$. Let $<<$ be an order on $T^*$ defined by $w <<v$ if either $w$ has smaller length than $v$ or both $w$ and $v$ have the same length and $w < v$.
For $u, w \in T^{\#}$ we set $u < w$ (resp. $u << w$) if the same holds for the associate words obtained from $u,w$ after forgeting brackets.

 A Lyndon-Shirshov monomial is a non-empty word $w$ in $T^{\#}$ such that
 either $w \in T$ or $w = w_1 w_2$ such that $w_1 > w_2$, $w_1$, $w_2$ are Lyndon-Shirshov monomials and if $w_1 = v_1 v_2$ then $v_2 \leq w_2$.  
 
 Denote by $K \langle T \rangle $ the free associative algebra generated by $T$ and by $K \langle T \rangle ^{(-)}$ the Lie algebra with underlying set $K \langle T \rangle $ and  with Lie operation $[x,y] = xy -  yx$. Define $L\langle T \rangle$ the  Lie subalgebra of $K \langle T \rangle ^{(-)}$  generated by  $T$. Note that $L\langle T \rangle$ is a free Lie algebra with a free basis $T$.
 
 \begin{theorem} \cite{B-K} \label{free} The set of Lyndon-Shirshov monomials forms a linear basis of $L\langle T \rangle $.
 \end{theorem}
 
 For   $f \in K \langle T \rangle \setminus \{ 0 \}$ we write $f$ as a linear combination (with coefficientes in $K$) of elements of $T^*$ and write $\overline{f}$ for the higher associative word with respect to $<<$ that is  in $f$.

 \medskip
2)  Let $A$ be a Lie algebra, $B$ a Lie subalgebra of $A$ and $d : B \to A\hbox{  a derivation }$ i.e. a linear map such that for $a, b \in B$ we have
$d([a,b]) = [d(a), b] + [a, d(b)].$
Recall that  by \cite{L-Sh}, \cite{W} a HNN extension  Lie algebra is given by the presentation in terms of generators and relations
 \begin{equation} \label{eq-HNN} L =  \langle A, t \ | \ [t,b] = d(b), b \in B \rangle \end{equation} 
 In \cite{L-Sh} the case of restricted Lie algebras is treated in more details and in \cite{W} the case of ordinary Lie algebras is considered.
In analogy with the HNN construction in group theory we call $A$ the base of the HNN extension, $B$ the associated Lie subalgebra and $t$ the stable letter.
   By \cite{W} $A$ embeds in $L$.

 Let $X$ be a basis of $A$ as a linear space that contains a basis $B_0$ of $B$. We consider a linear order $< $ on $T = X \cup \{ t \}$ such that for $b \in B_0, x \in X \setminus B_0$ we have that
 $b < x < t.$
The following results should be considered as a normal form for elements of HNN extension Lie algebra.

\begin{cor} \label{shirshov}  \cite{W} Let $L$ be the HNN extension  Lie algebra defined by (\ref{eq-HNN}).
The images of the Lyndon-Shirshov monomials $f$ from the free Lie algebras with free generating set $T = X \cup \{ t \}$ in $L$ such that $\overline{f}$ does not contain as a subword an element from  the set $\{ xy \ | \ x, y \in X, x > y \}  \cup \{ tb  \ | \ b \in B_0 \}$ form a basis of $L$ as a linear space (over $K$). 
\end{cor}
\iffalse{Example. Consider  $A = K a_1 \oplus K a_2$ a Lie algebra, $B = K a_1$ and $d : B \to A$ the inclusion map. Consider $T = \{ a_1, a_2, t \}$ with order $ a_1 < a_2 < t$. Suppose that $f$ is a Lyndon-Shirshov monomial  that contains $a_1$ and  $\overline{f}$ does not contain as a subword  $ a_2 a_1$ and $ t a_1$, hence $\overline{f} = a_1^m w$ where $w \in T^*$ does not contain $a_1$. 

Suppose $f \not= a_1$. Note that $f$ is obtained from $a_1^m w$ by choosing where to place the brackets. The definition of Lyndon-Shirshov monomial implies that $f = [a_1,[a_1, [\ldots,[ a_1, [w]]...]$, where $[w]$ is a Lyndon-Shirshov monomial with $\overline{[w]} = w$. But $[ a_1, [w]]$ is not a Lyndon-Shirshov monomial, a contradiction.
Thus  either $f = a_1$ or $f$ does not contain $a_1$ and there is no other restriction on $f$, hence $f$ is any  Lyndon-Shirshov monomial on the alphabet $\{ a_2, t \}$. Then by Corollary \ref{shirshov} and Theorem \ref{free} the HNN-extension  (\ref{eq-HNN}) has a decomposition as a vector space $L =  L \langle a_2, t \rangle \oplus K a_1$, where $ L \langle a_2, t \rangle $ is the free Lie algebra with a free basis $\{ a_2, t \}$.}\fi

\begin{cor} \label{product} Let $L$ be the HNN extension  Lie algebra defined by (\ref{eq-HNN}). Then 

a) $A \cap Kt = 0$;

b) if $A_0$ is a Lie subalgebra of $A$ then the subalgebra of $L$ generated by $A_0$ and $t$ is an HNN Lie extension
$\langle A_0, t \ | \ [t,b_0] = d(b_0) \hbox{ for } b_0 \in A_0 \cap B  \rangle$.
\end{cor} 
 \begin{proof}
 a) By the previous corollary the image of $T = X \cup \{ t \}$ in $L$ is a linearly independent set;
 
 b) Choose a linear basis $X_0$ of $A_0$ that contains a basis $Z_0$ of $B \cap A_0$ and extend $X_0$ to a linear basis $X$ of $A$ and extend $Z_0$ to a linear basis $B_0$ of $B$ such that $B_0 \subseteq X$.
 Then Corollary \ref{shirshov} implies that the embedding of $A_0$ in $A$ induces the embedding of the HNN Lie extension $\langle A_0, t \ | \ [t,b_0] = d(b_0) \hbox{ for } b_0 \in A_0 \cap B \rangle$  in the HNN Lie extension $\langle A, t \ | \ [t,b]  = d(b) \hbox{ for } b \in  B \rangle$  that preserves $t$.
 \end{proof}
 
 \subsection{Amalgamated products of Lie algebras} For $A$ and $B$ Lie algebras with a common Lie subalgebra $C$ we denote by $A \coprod_C B$ the amalgamated product of $A$ and $B$ with an amalgam $C$, see \cite[Chapter 4]{B-K}. Here we identify $A$ and $B$ with their images in $A \coprod_C B$, $A \coprod_C B$ is generated by $A$ and $B$ and the amalgamated product  is defined by the obvious universal property: for any Lie algebra $H$ together with homomorphisms of Lie algebras $\alpha : A \to H$ and $\beta : B \to H$ such that $\alpha |_C = \beta |_C$ there is an unique homomorphism of Lie algebras $\rho : A \coprod_C B \to H$ whose restriction to $A$ and $B$ is $\alpha$ and $\beta$ respectively. 
 
 \section{More on HNN extensions of Lie algebras}

In this section we prove an auxiliar new result on HNN Lie extensions.

\begin{proposition} \label{HNN}
Let $L = \langle A, t \ | \ [B, t] \subseteq B \rangle$ be an HNN extension Lie algebra and $M$ be an ideal of $L$ such that $L/M$ is one dimensional, $B \not\subseteq M$ and $M \cap A$ is finitely generated as a Lie algebra. Then

a) $M$ is finitely generated as a Lie algebra;

b) if  furthermore $M$ is finitely presented (in terms of generators and relations) we can conclude that $M \cap B$ is finitely generated.
\end{proposition}
 
\begin{proof}
a) We will show first that if $t \not\in M$ then we can change the stable letter $t$ to another one $t_0 \in M$. Note first that if $ ((K t + B) \setminus B)  \cap M \not= \emptyset $ then we can take $t_0 = k t + b \in   M$ for some  $ k \in K \setminus \{ 0 \},  b \in B$. Then
$[B, t_0] = [B, k t + b] \subseteq [B, kt] + [B, b] = k[B,t] + [B,b] \subseteq k B + B = B$.
Then $L$ is an HNN extension Lie algebra $\langle A, t_0 \ | \ [B, t_0] \subseteq B \rangle$.

Suppose that $(K t + B) \cap M = B \cap M = : C$. Since $B \not\subseteq M$ and $L/M$ is one dimensional, there is $b_0 \in B \setminus \{ 0 \}$ such that $B = K b_0 + C$ and so 
$C = (K t + B) \cap M  = (K t + K b_0 + C) \cap M = ((K t + K b_0) \cap M) + C,$ $ \hbox{ hence } (K t + K b_0) \cap M \subseteq C \subseteq B$.
Since  $t \not\in A, b_0 \in B \subseteq A$ we have $dim_K(K t + K b_0) = 2$ and $dim_K(L/ M) = 1$. We conclude that $(K t + K b_0) \cap M \not= 0$, hence  there are $k_1, k_2 \in K$ not both zero such that $k_1 t + k_2 b_0 \in (K t + K b_0) \cap M \subseteq C \subseteq  B$. If $k_1 \not= 0$ then  $k_1 t + k_2 b_0  \in B$ implies that $t \in B$, a contradiction. If $k_1 = 0, k_2 \not= 0$ then $k_2 b_0 = k_1 t + k_2 b_0 \in C$ implies $b_0 \in C$, hence $B = K b_0 + C = C \subseteq M$, a contradiction with $B \not\subseteq M$.

\iffalse{If  $(K t + B) \cap M = 0$ then 
$$Kt + B \simeq (Kt + B)/ ( (Kt + B) \cap M ) \hbox{ embeds in } L/M.$$ 
Since $dim_K(L/ M) = 1$ and $B \not\subseteq M$ we conclude that the image of $B$ in $L/M$ is the whole $L/M$, hence the image of $Kt + B$ in $L/M$ is the whole $L/M$ i.e. $$dim_K(Kt + B) = 1.$$Since $dim_K (Kt) = 1 = dim_K(Kt + B)$ we conclude that $Kt = Kt + B$, so $B \subseteq Kt$. Note that $B \subseteq A$ and by the normal form in HNN extension Lie algebras we have that $A \cap Kt = 0$, hence $$B \subseteq Kt \cap A = 0 \subseteq M,$$
a contradiction with $B \not\subseteq M$.}\fi

By substituting $t$ with $t_0$ if necessary we can assume from the very beginning that $t \in M.$
Let $L_0$ be the Lie subalgebra of $L$ generated by $A \cap M$ and $t$.

\medskip
{\bf Claim 1} $L_0$ is an ideal of $L$.

\medskip
Since $dim_K L/ M = 1$ we conclude that $dim_K (A/ A \cap M) \leq 1$, otherwise $A \cap M = A$ then $B \subseteq A \subseteq M$, a contradiction since $B \not\subseteq M$. Furthermore $B \not\subseteq A \cap M$ implies that $A = M \cap A + B \subseteq  \langle M \cap A, B \rangle \subseteq A, \hbox{ hence } A = \langle M \cap A, B \rangle.$
Then
$L = \langle A, t \rangle = \langle M \cap A, B, t \rangle.$
Note that $\langle M \cap A, t \rangle = L_0$, hence  
$L = \langle L_0, B \rangle$.
implies that to prove Claim 1 is equivalent to show the following claim 

\medskip
{\bf Claim 2}
$[L_0, B ] \subseteq L_0$.

\medskip
This is equivalent to showing that for every $\lambda = [\lambda_1, \ldots, \lambda_n]$ a left normed commutator where $\lambda_1, \ldots, \lambda_n \in (A \cap M) \cup \{ t \}$ and $b \in B$ we have
$[\lambda, b ] \in L_0.$
To show this we induct on $n$.

First suppose that $n = 1$. If $\lambda_1 \in A \cap M$ then using that $A \cap M$ is an ideal of $A$ we have
$[\lambda_1, b] \subseteq [A \cap M, b] \subseteq A \cap M \subseteq L_0.$
If $\lambda_1 = t$ then using that $t \in M$
$[\lambda_1, b] = [t,b] \subseteq B \cap [M, B] \subseteq B \cap M \subseteq A \cap M \subseteq L_0.$

For $n > 1$ we have
$[\lambda, b] = [\lambda_1, \ldots, \lambda_n, b] =
[\lambda_1, \ldots, \lambda_{n-1}, b, \lambda_n] + [\lambda_1, \ldots, \lambda_{n-1},$ $[ \lambda_n,b]]$.
By the case $n = 1$ we have that $[\lambda_n, b] \in L_0$, hence by induction on $n$
$[\lambda_1, \ldots, \lambda_{n-1},[ \lambda_n,b]] \subseteq [\lambda_1, \ldots, \lambda_{n-1}, L_0] \subseteq L_0$.
\iffalse{
$$
[L_0, [\lambda_n,b]] \subseteq [L_0, L_0] \subseteq L_0.$$}\fi
By induction $[\lambda_1, \ldots, \lambda_{n-1}, b] \in L_0 $, hence
$[\lambda_1, \ldots, \lambda_{n-1}, b, \lambda_n] \in [L_0, \lambda_n] \subseteq [L_0, L_0] \subseteq L_0.$
This completes the inductive step of the proof of Claim 2 and completes the proofs of both claims.

\medskip
Then
$ L/ L_0 \simeq \langle A, t \rangle / \langle A \cap M, t \rangle \simeq A / (A \cap M) \simeq K$
i.e. $dim_K L/ L_0 = 1$. This combined with the fact that $L_0 \subseteq M \subseteq L$ and $dim_K L/ M = 1$ implies that $M = L_0 = \langle A \cap M, t \rangle$ is finitely generated as required.

b) Note that $M = L_0$ and by construction
$L_0 = \langle A \cap M, t \rangle.$
Thus by Corollary \ref{product} b) $L_0$ is itself an HNN extension Lie algebra with stable letter $t$ and base Lie subalgebra $A \cap M$ i.e.
$$
L_0 = \langle A \cap M, t \ | \ [B \cap M, t] \subseteq B \cap M \rangle.
$$
 
 Suppose now that $M$ is finitely presented and $B \cap M$ is not finitely generated.  Let $ \{ b_i \ | \ i\geq 1 \}$ be an infinite generating set of $B \cap M$. Then we define for an integer $k \geq 1$
 $$M_k = \langle A \cap M, t \ | \ [b_i,t] = \widetilde{b}_i \hbox{ for }  i \leq k \rangle$$
  where $[b_i,t] = \widetilde{b}_i $ in $M$. 
  Note that there is an epimorphism $\pi_k : M_k \to M_{k+1}$ that is not an isomorphism. Then $\{ M_k, \pi_k \}_k$ is a direct system of Lie algebras with direct limit $M$. Since $M$ is finitely presented, say with a finite set of relations $R$, then for sufficiently big $k_0$ we have that the elements of $R$ are relations in $M_{k_0}$, hence the canonical map $M_{k_0} \to M$ is an isomorphism. This implies that each $\pi_k$ for $k \geq k_0$ is an isomorphism, a contradiction.
\end{proof}

\section{On codimension 1 ideals via the Noether normalization theorem}

Recall from commutative algebra that if $A$ and $B$ are commutative rings with unity and $B$ is a subring of $A$ we say that $A$ is integral over $B$ if
 every $a \in A$ is integral over $B$ i.e. a root of a monic polynomial with coefficients in $B$. If $A = B[a_1, \ldots, a_s]$ then 
 $A$ considered as a $B$-module (via multiplication) is finitely generated if and only if each $a_i$ is integral over $B$. 
The following is a version of the Noether normalization theorem.

\begin{theorem} \cite[Ch. II,~Thm.3.1]{Ku} \label{Noether}
Let $K$ be an infinite field and $A = K[y_1, \ldots, y_m]$ a finitely generated commutative $K$-algebra. Then there exist $x_1, \ldots, x_n \in V = Ky_1 + \ldots + K y_m$ such that $B = K[x_1, \ldots, x_n]$ is a polynomial ring on $n$ variables for some $n \geq 0$ and $A$ is an integral extension of $B$.
\end{theorem}

In the above theorem $n$ is the Krull dimension of the ring $A$.
We will need the following simplified version of the above theorem.

\begin{cor}  \label{cor-noether}Let $K$ be an infinite field and $A = K[y_1, y_2]$ a finitely generated commutative $K$-algebra of Krull dimension $n \leq 1$.
Suppose that $A$ is finitely generated as $B_0$-module, where $B_0 = K[y_1]$. Then either 

1) $A$ is finite dimensional over $K$ or 

2) $B_0$ is a polynomial ring and there is a finite subset $K_0 \subset K$ such that for any $\lambda \in K \setminus K_0$ and  $x_1 = \lambda
 y_1 + y_2$,  we have that $B = K[x_1]$ is a polynomial ring  and $A$ is an integral extension of $B$.
\end{cor}
 
\begin{proof}
If $n = 0$ then $A$ is finite dimensional.

Suppose $n = 1$.  Since $A$ and $B_0$ have the same Krull dimension, we conclude that  $B_0$ has Krull dimension 1, hence $B_0$ is  a polynomial ring.

 Let
$
f(t) = t^d + b_{d-1} t^{d-1} + \ldots + b_{i} t^{i} + \ldots +  b_0 \in B_0[t]
$
be a monic non-zero polynomial of smallest possible degree $d$ such that $f(y_2) = 0$ in $A$.
Thus
$ y_2^d + b_{d-1} y_2^{d-1} + \ldots + b_{i}  y_2^{i} + \ldots + b_0  = 0
$
in $A$. Substitute $y_2 = x_1 - \lambda y_1$ above for $\lambda \in K \setminus \{ 0 \}$ and note that $A = K[x_1, y_1]$.  Thus we have
$$(x_1 - \lambda y_1)^{d} + b_{d-1}(x_1 - \lambda y_1)^{d-1} + \ldots + b_{i}(x_1 - \lambda y_1)^{i}
+ \ldots + b_0  = 0$$
and it can be written ( after dividing by an appropriate element of $K$) as $y_1$ satisfying a monic polynomial over $K[x_1]$ if 
the part of the above equation that depends only on $y_1$ is non-zero i.e. 
 $$g(y_1) = ( - \lambda y_1)^{d-1} + b_{d-1}( - \lambda y_1)^{d-2} + \ldots + b_{i}( - \lambda y_1)^{i}
+ \ldots + b_1   \not=  0$$
is a non-zero polynomial, and $deg(g) = max_{0 \leq i \leq d} \{ deg (b_i ( -\lambda y_1)^{i-1}) \}$ where $b_d = 1$. For this to hold we have to avoid the roots of finitely many polynomials i.e.  this is true for all but finitely many elements $\lambda$ from $K$.
\end{proof}

\begin{proposition} \label{sigma-Lie}
Let $L$ be an $\mathbb{N}$-graded Lie algebra over an infinite field $K$ with a graded ideal $N$ such that  $L/ N$ is abelian and  $dim_K L/ N = 2$. Suppose furthermore that $L = N \oplus K x_1 \oplus K x_2$ where the direct sum is of vector spaces, $x_1$ and $x_2$ are homogeneous elements of the same degree and $M = N \oplus K x_1$ is finitely generated as a Lie algebra. Then there exists a $\mathbb{N}$-graded  ideal $S_0$ of $L$ such that

a) $S_0 \not= M$,  $N \subseteq S_0$, $dim_K  L/ S_0 = 1$;

b) $S_0$ is finitely generated as a Lie algebra.

\end{proposition}  

\begin{proof}
Recall that an $\mathbb{N}$-graded Lie algebra is finitely generated if and only if its abelianization is finitely generated i.e. is finite dimensional. For more details on this and other properties of $\mathbb{N}$-graded Lie algebras the reader can check the preliminaries of \cite{We}.

Thus to find  a $\mathbb{N}$-graded  ideal $S_0$  that is finitely generated as a Lie algebra  we need only to prove  that $S_0/ [S_0, S_0]$ is finite dimensional. Thus it suffices to show that the image of $S_0$ in $L/ I$ is finitely generated, where $I$ is some fixed graded ideal of $L$ that is contained in $[S_0, S_0]$. We fix $I = [N,N]$. Thus we can assume from the very beginning that $N$ is abelian.

Note that the universal enveloping algebra $U(L/ N) $ can be identified with a commutative polynomial ring $K[z_1, z_2]$, where $L/ N = K z_1 \oplus K z_2$, each $z_i$ is the image of $x_i$ in $L/ N$ and $M/ N = K z_1$. Consider the commutative $K$-algebra
$$A = K[z_1, z_2]/ ann_{K[z_1, z_2]} (N)$$
Here we view $N$ as a $U(L/N)$-module via the adjoint action of $L$, hence as a $ K[z_1, z_2]$-module  and $ann$ means the annihilator i.e. the elements that act as zero.

Let $y_1, y_2$ be the images of $z_1, z_2$ in $A$. Thus $A$ is a finitely generated $K$-algebra with generators $y_1$ and $y_2$.
We set $B = K[y_1]$, a subalgebra of $A$.

\medskip
{\bf Claim} $A$ is an integral extension of $B$.

\medskip
Indeed since $L$ is a finitely generated Lie algebra  and $L/ N$ is a finitely presented Lie algebra, we conclude that $N$ is finitely generated as a module over $U(L/ N)$ via the adjoint action i.e. $N$ is a finitely generated as an $A$-module. Since $M$ is finitely generated as a Lie algebra we have that $N$ is finitely generated as a right  $B$-module, i.e. for some $n_1, \ldots, n_s \in N$ we have   $$N = n_1 B + \ldots + n_s B$$ where $n_i b$ is the element obtained from $n_i$ after applying the adjoint action of $b$. Consider the $s \times s$-matrix $T = (b_{ij})$, where $b_{ij} \in B$ and $n_i y_2 = \sum_j n_j b_{ji}$. Then $T$ satisfies its monic Hamilton-Kelly polynomial ( the characteristic poly) with coeficients in $B$, hence $y_2$ is integral over $B$. Since $A = B[y_2]$ we conclude that $A$ is finitely generated over $B$. This completes the proof of the claim.

\medskip
If $B$ is finite dimensional over $K$ then $A$ is finite dimensional over $K$, hence any codimension  one ideal $S_0$ will work.

Suppose that $B$ is infinite dimensional over $K$. Then $B$ is a polynomial ring in one variable. Since $A$ is finitely generated over $B$ and integral extension does not change the Krull dimension, we conclude that  $A$ has Krull dimension 1. Then by  Corollary \ref{cor-noether} we can choose $S_0$ such that $S_0/ N = K t,$ where $t \in  K y_1 + y_2$ and the subring $K[t] $ of $A $ is a polynomial ring in one variable with $A$  finitely generated over $K[t].$
Then there exist $a_1, \ldots, a_m \in A$ such that $A = a_1 K[t] + \ldots + a_m K[t]$
as a $K[t]$-module. Then
$N = n_1 B + \ldots + n_s B = n_1 A + \ldots + n_s A = \sum_{1 \leq i \leq s, 1 \leq j \leq m} n_i a_j K[t]$
is a finitely generated $K[t]$-module, hence $S_0$ is finitely generated as a Lie algebra.
\end{proof}

\section{Proof of Theorem A and Corollaries B1 and B2}

{\bf Proof of Theorem A}  Let $s_1, \ldots, s_n$ be a free basis of homogeneous elements of  $F_2$. Consider $$L_i = F_1 \leftthreetimes \langle s_i \rangle$$ where the adjoint action of $s_i$ on $F_1$ in $L_i$ is the adjoint action of $s_i$ on $F_1$ in $L$ i.e. $L_i$ is the subalgebra of $L$ generated by $F_2$ and $s_i$. Then we have a decomposition as a free amalgamated product of Lie algebras
$$L = L_1 \coprod_{F_1} L_2 \coprod_{F_1} \ldots \coprod_{F_1} L_n$$

Recall that
$N \subseteq F_1 \subseteq L_1$
with $dim_K F_1/ N = 1 = dim_K L_1/ F_1$ and $F_1$ is finitely generated as a Lie algebra. Note that $[L_1, L_1] \subseteq F_1 \cap [L,L] \subseteq N,$
hence $L_1/ N$ is abelian. Note that $L_1$ is $\mathbb{N}$-graded with $N$ and $F_1$ graded subalgebras and  $L_1 = N \oplus K y \oplus K s_1$ with $y \in F_1$, $s_1$ and $y$ homogeneous elements of the same degree.  

By Proposition \ref{sigma-Lie} there is an ideal $S_0$ of $L_1$ such that
$N \subseteq S_0 \subseteq L_1$
with $dim_K S_0/ N = 1 = dim_K L_1/ S_0$,  $S_0 \not= F_1$ and $S_0$ is finitely generated as a Lie algebra.
By construction $S_0 = N \leftthreetimes \langle s_1 + \lambda y \rangle $
where $F_1 = N \leftthreetimes \langle y \rangle$ and  $\lambda$ could be any element from  $ K \setminus K_0$, for some finite subset $K_0$ of $K$. Since $K$ is infinite we can assume that $\lambda \not= 0$.

Since $L = F_1 \leftthreetimes F_2$ we have
$[L,L] \cap L_1 = ([F_1, F_1] + [F_1, F_2] + [F_2, F_2]) \cap (F_1 + K(s_1 + \lambda y )) = ([F_1, F_1] + [F_1, F_2]) $ $+ ([F_2, F_2] \cap K(s_1 + \lambda y )) =  [F_1, F_1] + [F_1, F_2] \subseteq [L,L] \cap F_1 \subseteq N \subseteq S_0.$
Since  $[L,L] \cap L_1 \subseteq S_0$ and $1 = dim_K L_1/ S_0$, there is an epimorphism of Lie algebras
$\mu : L \to K$ such that $Ker(\mu) \cap L_1 = S_0$. Hence
$Ker(\mu) \cap F_1 = (Ker(\mu) \cap L_1) \cap F_1 = S_0 \cap F_1 = N.$
Consider the Lie algebra
$$L_1 \coprod_{F_1} L_2 = \langle L_1, s_2 \ | \ [F_1, s_2] \subseteq F_1 \rangle$$
Note it is an HNN extension Lie algebra.
By Proposition \ref{HNN} $Ker(\mu) \cap (L_1 \coprod_{F_1} L_2)$ is finitely generated if $Ker(\mu) \cap L_1 = S_0$ is finitely generated.

Consider the Lie algebra
$$L_1 \coprod_{F_1} L_2 \coprod_{F_1} L_3 = \langle L_1 \coprod_{F_1} L_2, s_3 \ | \ [F_1, s_3] \subseteq F_1 \rangle$$
By Proposition \ref{HNN}  $Ker(\mu) \cap (L_1 \coprod_{F_1} L_2 \coprod_{F_1} L_3)$ is finitely generated if $Ker(\mu) \cap (L_1 \coprod_{F_1} L_2) $ is finitely generated.

Continuing in the same fashion we go up to the HNN Lie algebra $$L = L_1 \coprod_{F_1}  \ldots \coprod_{F_1} L_n = \langle L_1 \coprod_{F_1}  \ldots \coprod_{F_1} L_{n-1}, s_n \ | \ [F_1, s_n] \subseteq F_1 \rangle$$
Then by Proposition \ref{HNN} $Ker(\mu)$ is finitely generated if $Ker(\mu) \cap (L_1 \coprod_{F_1}  \ldots \coprod_{F_1} L_{n-1})$ is finitely generated. Thus $Ker(\mu)$ is finitely generated.
Furthermore by Proposition \ref{HNN} if $Ker(\mu)$ is finitely presented then $Ker(\mu) \cap F_1 = N$ is finitely generated, a contradiction. Hence $Ker(\mu)$ is not finitely presented. Finally set $L_0 = Ker (\mu)$.

\medskip
{\bf Proof of Corollary B1} We will show that we can apply Theorem A and construct an ideal $L_0$ of $L$ of codimension 1 that is finitely generated but not finitely presented. In the case when $F_1$ is free, the fact that $L_0$ is not finitely presented has an alternative proof using homological methods, see Lemma \ref{Euler0}.

Since
 $F_1 \not\subseteq [L,L]$ we choose $N$ to be an  ideal of $F_1$ of codimension 1 such that $[L,L] \cap F_1 \subseteq N$.
Since the grading of $F_1$ assigns the same degree $n_0$ to the elements of a standard generating set of $F_1$ we conclude that any codimension 1 ideal of $F_1$ is homogeneous i.e. is $\mathbb{N}$-graded, in particular $N$ is homogeneous  and furthermore $F_1 = N \oplus K y$, where $y$ is a homogeneous element of degree $n_0$. 

Note that if $F_1$ is free it follows by the main result of \cite{B.Baumslag} that $N$ is not finitely generated. We claim that if $F_1$ is a surface Lie algebra then $N$ is not finitely generated, Then we can apply Theorem A.

 To prove the above claim assume the contrary,  $N$ is finitely generated in the surface case. Then for $Q_0 = N/ [F_1, F_1]$ $$V = [F_1, F_1]/ [[F_1, F_1], [F_1, F_1]]$$  considered as a right $U(Q_0)$-module via the adjoint action is finitely generated.
Set $Q = F_1/ [F_1, F_1]$. Note that  $$F_1 = \langle x_1, \ldots, x_{2n} \ | \ [x_1, x_2] + \ldots + [x_{2n-1}, x_{2n}] = 0 \rangle,$$ hence  $V$ is generated as a $U(Q)$-module by the images $a_{i,j}$ of $[x_i, x_j]$ modulo the relation $a_{1,2} + a_{3,4} + \ldots + a_{2n-1, 2n} = 0$ and the Jacobi relations $a_{i,j} \circ q_k + a_{j,k} \circ q_i + a_{k,i} \circ q_j = 0,$  where $q_i$ is the image of $x_i$ in $Q$, and $a_{i,j} = - a_{j,i}$. Here $\circ$ denotes the adjoint action i.e. $a_{i,j} \circ q_k $ is the image of $[[x_i, x_j], x_k]$ in $V$.

 We can consider $S$ the field of fractions of $U(Q)$ and set $W = V \otimes_{U(Q)} S$. Then using the above relations (Jacobi and antisimmetry) we can express the image $b_{i,j}$ of $a_{i,j}$ in $W$ as element of  the $S$-submodule generated by $b_{1,i}$ and $b_{1,j}$. Thus $W$ is generated as an $S$-module ( this is as a vector space over $S$) by $2n-1$ elements $\{ b_{1,i} \ | \  2 \leq i \leq 2n \}$ modulo only one relation (the defining relation of the surface Lie algebra i.e. $a_{1,2} + a_{3,4} + \ldots + a_{2n-1, 2n} = 0$). Thus $dim_S W \geq 2n -1 -1 = 2n - 2 > 0,$ hence $V$ contains a free $U(Q)$-submodule i.e. a submodule isomorphic to $U(Q)$. Thus if $V$ is finitely generated as a $U(Q_0)$-module, by the Noetherianess of $U(Q_0)$ we deduce that every submodule is finitely generated, in particular $U(Q)$ is finitely generated as a $U(Q_0)$-module ( via the multiplication), a contradiction since for $Q = Q_0 \oplus K q$.

\medskip
{\bf Proof of Corollary B2}

Let $V = (F_2)_1$ the vector space of the elements of $F_2$ of degree 1 plus the zero element. Consider the linear map
$$\theta: F_2 \to End_K(F_1/ [F_1, F_1])$$
induced by the adjoint action of $F_2$ on $F_1$ i.e.
$\theta(w) (f + [F_1, F_1]) = [f,w] + [F_1, F_1],$ where $End_K(F_1/ [F_1, F_1])$ is the vector space of all linear endomorphisms of $F_1/ [F_1, F_1]$ i.e. all linear maps from  $F_1/ [F_1, F_1]$ to  $F_1/ [F_1, F_1]$.
Note that for $n = d(F_1)$ we have that $F_1/ [F_1, F_1]$ is a vector space ( over $K$) of degree $n$, hence $End_K(F_1/ [F_1, F_1]) \simeq M_n(K)$. 
For $m = d(F_2) > d(F_1)^2 = n^2$ we have that $dim_K V = m$, hence $Ker(\theta) \cap V \not= 0$.

Note that every basis of $V$ as a vector space (over $K$) is a free basis of $F_2$. Take $s_1 \in Ker(\theta) \cap (V \setminus \{ 0 \})$ and  $b \in [F_2, F_2] \cap (Ker(\theta) \setminus \{ 0 \}) $ be a homogeneous element. 
Define $\widetilde{F}_2$ to be the Lie subalgebra of $F_2$ generated by $s_1$ and $b$. As a Lie subalgebra of a free one is free and obviously $\widetilde{F}_2$ is not 1 dimensional we conclude that $s_1$ and $b$ is a free basis of $ \widetilde{F}_2$. As the generators $s_1$ and $b$ are homogeneous $\widetilde{F}_2$ is a graded subalgebra of $F_2$.

Consider the Lie algebra $\widetilde{L} = F_1 \leftthreetimes \widetilde{F}_2$. It is generated by homogeneous elements, hence it is a graded Lie subalgebra of $L$. We want to apply Corollary B1 for $\widetilde{L}$ and $n_0 = 1$, and conclude that $\widetilde{L}$ is not coherent, hence $L$ is not coherent.
We need to show that $F_1 \not\subseteq [\widetilde{L}, \widetilde{L}]$ i.e. $F_1 \not\subseteq [\widetilde{L}, \widetilde{L}] \cap F_1  = [F_1, F_1] +
  [F_1, \widetilde{F}_2]$.
  By construction $\widetilde{F}_2 \subseteq Ker(\theta)$, hence $ [F_1, \widetilde{F}_2] \subseteq [F_1, F_1]$. This completes the proof.

\section{More on finitely generated Lie subalgebras}

\begin{proposition} \label{gen1}
Let $0 \to S \to L \to \Gamma \to 0$ be a short exact sequence of $\mathbb{N}$-graded Lie algebras over an infinite field $K$. Suppose that $S$ has a graded ideal $N$ such that 

a) $[L,L] \cap S \subseteq N$,

b) there are homogeneous elements $x$ and $s_1$ of the same degree such that $S = N + k x$ and $s_1 \in \Gamma \setminus [\Gamma, \Gamma]$.

 Then there is an ideal $M$ of $L$ such that $M$ is finitely generated as a Lie algebra, $dim_K (L/ M) = 1$ and $M \cap S = N$. 
\end{proposition}

\begin{proof}
Consider a commutative diagram 
$$\xymatrix{S \ \ar@{^{(}->}[r] \ar[d]_{id_S} & \widetilde{L} \ar@{->>}[r]^{} \ar@{->>}[d]_{\pi} & F_n \ar@{->>}[d]^{} \\ S \ \ar@{^{(}->}[r] & L \ar@{->>}[r]^{} & \Gamma}$$
where the lines are short exact sequences of $\mathbb{N}$-graded Lie algebras,
$F_n$ is the free Lie algebra with a homogeneous free basis $s_1, \ldots, s_n$, where  $s_1, \ldots, s_n$ is a homogeneous generating set of $\Gamma$ and the vertical maps are surjective homomorphisms of graded Lie algebras with the most left map being the identity. 
Here $$\widetilde{L} = L_1 \coprod_S L_2 \coprod_S \ldots \coprod_S L_n,$$ where $\coprod_S$ is the amalgamated free product with amalgam $S$ in the category of $\mathbb{N}$-graded Lie algebras, and each $L_i = S \leftthreetimes
 \langle s_i \rangle$
 the subalgebra of $L$ generated by $S$ and $s_i$. The restriction of $\pi$ on $L_i$ is the identity map.

 Note that  $[N, s_1] \subseteq [S, s_1] \subseteq [L_1,L_1] \cap S  \subseteq [L,L] \cap S\subseteq  N,$ hence  $N$ is an ideal of $L_1$ and $[L_1,L_1] = [S + k s_1, S + k s_1] = [S,S] + [S, s_1] \subseteq N$. We have $N \subseteq S \subseteq L_1$ where $dim_k (S/ N) = 1$, this together with the inclusion $[L_1, L_1]\subseteq N$ implies that $L_1 / N $ is an abelian Lie algebra with $dim_K(L_1/ N) = 2$. 

By Proposition \ref{sigma-Lie} there is $S_0$ an ideal of $L_1$ such that $N \subseteq S_0, S_0$ is finitely generated as a Lie algebra, $S_0 \not= S$ and $dim_K(L_1/ S_0) = 1.$
Let $\mu : L \to K$ be a homomorphism of Lie algebras such that $Ker (\mu \circ \pi) \cap L_1 = S_0, \hbox{ i.e. } Ker(\mu) \cap  L_1 = S_0.$ This is possible since  $s_1 \not\in [\Gamma, \Gamma]$ and $[L,L] \cap S \subseteq N \subseteq S_0$. Note that $S \not\subseteq S_0$, hence  $\mu(S) \not= 0$.

Consider the epimorphism of Lie algebras $\chi = \mu  \circ \pi : \widetilde{L} \to K.$ Note that $\chi(S) \not= 0, Ker(\chi) \cap L_1 = S_0$ is finitely generated  and $Ker(\chi) \cap S = S_0 \cap S = N.$ Then we view $L_1 \coprod_S L_2$ as an HNN Lie extension $\langle L_1, s_2 \ | \ [S,{s_2}] \subseteq S \rangle$ with a Lie base $L_1$, associated Lie algebra $S$ and stable letter $s_2$. Then by Proposition \ref{HNN} a) $$Ker(\chi) \cap ( L_1 \coprod_S L_2)\hbox{ is finitely generated }.$$ We view $L_1 \coprod_S L_2 \coprod_S L_3$ as an HNN Lie extension with a base Lie subalgebra $L_1 \coprod_S L_2$, associated Lie subalgebra $S$ and stable letter $s_3$. Then by Proposition \ref{HNN} a) $$Ker(\chi) \cap ( L_1 \coprod_S L_2 \coprod_S L_3)\hbox{ is finitely generated }.$$ Then repeating this argument several times we deduce that $Ker(\chi)$ is finitely generated.
Note that $Ker(\mu)$ is a quotient of $Ker(\chi)$, hence $Ker(\mu)$ is finitely generated. Finally we set $M = Ker(\mu)$.

\end{proof}

\section{An example that the $\mathbb{N}$-graded condition in Proposition \ref{sigma-Lie} is not redundant}

Let $F$ be the free Lie algebra with basis $x,y$ and $N$ be the ideal generated by $x$. By \cite{Sh}, \cite{Witt} a Lie subalgebra of a free one is free. Here it is easy to point a free basis $ \{ a_i \ | \ i \geq 0 \}$ of $N$, where $a_i = [x, y , \ldots , y ]$ where $y$ appears $i$ times, all commutators are left-normed, thus $a_0 = x, [a_i,y] = a_{i+1}.$

\noindent
Consider the Lie algebra $L = F \leftthreetimes \langle s \rangle, \hbox{ where } [a_0, s] = [a_0, a_1], [y,s] = 0$.
Note that since $[y,s] = 0$ we have
$[a_i,s] = [x, y, \ldots, y, s] = [x,s, y, \ldots, y] = [a_0,a_1, y, \ldots, y] =: b_{i+1}$
Since the map $ad(y) :F \to F$ sending $f$ to $[f,y]$ is a derivation we have
$$b_{i+1} = \sum_{0 \leq j \leq i} {i \choose j} [ad(y)^j (a_0), ad^{i-j}(y) (a_1)]= \sum_{0 \leq j \leq i} {i \choose j} [a_j, a_{i+1-j}]$$
Let $M$ be a Lie subalgebra of $L$ that contains $N$ such that $M \not= F$. Then there is some $\lambda \in K$ such that for $t = \lambda y + s$ we have
$M = N \leftthreetimes \langle t  \rangle.$

\begin{lemma} \label{not-redundant} For any $\lambda \in K$ we have that $M$ is not finitely generated. Thus if $I$ is a Lie subalgebra of $L$ such that $dim_K L/ I = 1$, $N \subseteq I$ and $I$ is finitely generated then $I = F$.
\end{lemma}

\begin{proof}

If $\lambda = 0$ then consider the Lie algebra $L_0 = M/[N,N]$.
 Note that  the adjoint action of $t$ on $N/ [N,N]$ is the trivial one i.e. acts as the zero map. Note that  $J = N/ [N,N]$ as a vector space over $K$ has an infinite basis  $\{ a_i + [N,N] \ | \ i \geq 0 \}$ and $L_0/ J $ is one dimensional over $K$ with generator $t + J$. Thus $J$ is not finitely generated as $U(L_0/J)$-module, hence $L_0$ is not finitely generated. Then $M$ is not finitely generated.
 
 Suppose that $\lambda \not= 0$.  Note that
 $[a_0, t] = [a_0, \lambda y + s] = [a_0, \lambda y] + [a_0, s] = 
 \lambda a_1 + [a_0, a_1]$
 and in general since $[t,y] = 0$ we have
 $$[a_i, t] = ad(t)(ad(y)^i(a_0)) =  ad(y)^i ad(t) (a_0) = ad(y)^i ( \lambda a_1 + [a_0, a_1]) =$$ $$ \lambda a_{i+1} + ad(y)^i([a_0, a_1]) = \lambda a_{i+1} + b_{i+1} = \lambda a_{i+1}  + [a_0, a_{i+1}] + \sum_{1 \leq j \leq i} {i \choose j} [a_j, a_{i+1-j}] $$
Let $L_1$ be the quotient of $M$ obtained by killing the ideal generated by $\{t, [a_i, a_j] \ | \ 1 \leq i < j \}$. Then by the above relation we get that $L_1$ has the following presentation as a Lie algebra
$$
L_1 = \langle a_0, \ldots , a_i, \ldots \ | \ \lambda a_{i+1} + [a_0, a_{i+1}]   \hbox{ for } i \geq 0,  [a_i, a_j]  \hbox{ for } 1 \leq i < j  \}$$
Thus writing $\bar{a}_i$ for the image of $a_i$ in $L_1$  we have
\begin{equation} \label{eq}  - \lambda \bar{a}_{i+1} = [\bar{a}_0, \bar{a}_{i+1}] \end{equation}
Thus
$[L_1, L_1]$ is an abelian ideal of $L_1$ (hence $L_1$ is metabelian), $[L_1, L_1]$ as a vector space (over $K$) has a basis $\{ \bar{a}_i \ | \ i \geq 1 \}$ and $L_1 / [L_1, L_1]$ is one dimensional ( over $K$) with a generator the image of $\bar{a}_0$. By (\ref{eq}) the adjoint action of $\bar{a}_0$ is multiplication by $ \lambda$, hence $[L_1, L_1]$ is infinitely generated as a module over the universal algebra of $L_1/ [L_1, L_1]$. Then $L_1$ is not finitely generated, hence $M$ is not finitely generated.
\end{proof}

Note that it is impossible that the Lie algebras $M$ and $L$ are $\mathbb{N}$ graded, $x,y,s$ are homogeneous and $deg(y) = deg(s)$. Otherwise $a_0, a_1$ are homogeneous,  $deg(a_0) + deg(s) = deg ([a_0, s]) = deg ([a_0, a_1]) = deg(a_0) + deg (a_1)$ hence
$deg(y) = deg(s) = deg(a_1) = deg ([x,y]) = deg(x) + deg(y), \hbox{ hence } deg(x) = 0$
a contradiction.

\section{Examples of coherence} \label{sec-ex}

Let $L$ be an arbitrary Lie algebra over a field $K$. Consider the complex
$$
 \ldots \to \wedge^n L \mapnew{d_n} \wedge^{n-1} L \mapnew{d_{n-1}} \ldots \to L \to  K \to 0
$$
defined with the differentials
$$
d_n( x_1 \wedge \ldots \wedge x_n ) =
\sum_{1 \leq i < j \leq n}(-1)^{i+j} [x_i, x_j] \wedge x_1 \wedge \ldots \widehat{x_i} \ldots \widehat{x_j} \ldots \wedge x_n $$
The $n$th-homology of $L$ is
$
H_n(L,K) = Ker (d_n) / Im (d_{n+1}).
$

Let $0 \to N \to L \to Q \to 0$ be a short exact sequence of Lie algebras.
As for groups there is Lyndon-Hoschild-Serre spectral sequence
for any $L$-module $W$
$$E^2_{p,q} = H_p(Q, H_q(N, W))$$ that converges to $H_{p+q} (L, W)$.

\begin{lemma}
Let $L = Q \leftthreetimes F$ be a Lie algebra, where $F$ and $Q$ are finitely generated free Lie algebras and $dim_K Q = 1$. Then $L$ is coherent.
\end{lemma}

\begin{proof}
Let $L_0$ be a finitely generated  Lie subalgebra of $L$. Then $L_0 \simeq Q_0 \leftthreetimes F_0$, where $Q_0 = Q \cap L_0$ and $F_0 \simeq L_0/ Q_0 \simeq (L_0 + Q)/ Q \leq L/ Q \simeq F.$ Thus $F_0$ is finitely generated and free Lie algebra, hence finitely presented. Note that $dim_K Q_0 \leq 1$, hence $Q_0$ is finitely presented. An extension of a finitely presented Lie algebra by a finite presented one is finitely presented, so $L_0$ is finitely presented.
\end{proof}

Let $L$ be a $\mathbb{N}$-graded Lie algebra. We say that it is graded coherent if every graded Lie subalgebra is coherent.

\begin{proposition} \label{coherent0}
Let $L = F \leftthreetimes Q$ be a $\mathbb{N}$-graded Lie algebra with $F$ finitely generated free and $dim_K Q = 1$. Then $L$ is graded coherent.
\end{proposition}

\begin{proof}
Let $L_0$ be a graded finitely generated Lie subalgebra of $L$. If $L_0$ is a subalgebra of $F$ then it is free, hence is finitely presented.

If not $L_0 \simeq F_0 \leftthreetimes Q_0$, where $F_0$ is free and $dim_K Q_0 = 1$.
We aim to show that $L_0$ is finitely presented, and since $L_0$ is graded it is equivalent to $dim_K H_2(L_0, K) < \infty$.
Consider the spectral sequence $E_{i,j}^2 = H_i(Q_0, H_j(F_0, K)).$ Since both $F_0$ and $Q_0$ are free we have that $E_{i,j}^2 = 0$ for $i \geq 2$ or $j \geq 2$. Then the spectral sequence colapces i.e. all differentials are zero and $$H_2(L_0, K) \simeq E_{1,1}^2 = H_1(Q_0, H_1(F_0, K)) = H_1(Q_0, F_0/ [F_0, F_0])$$
Note that since $L_0$ is finitely generated and $Q_0$ is finitely presented we have that $F_0$ is finitely generated as an ideal of $L_0$. Hence the abelianization $F_0/ [F_0, F_0]$ is finitely generated as a $U(Q_0)$-module via the adjoint action of $Q_0$. Furthermore since $Q_0$ is one dimensional we have that
$ H_1(Q_0, F_0/ [F_0, F_0]) \simeq H^0(Q_0,  F_0/ [F_0, F_0]).$
Note that $F_0/ [F_0, F_0]$ is a finitely generated $U(Q_0)$-module and $U(Q_0)$ is a polynomial ring in one variable with coefficients in the field $K$, hence is PID. Then $F_0/ [F_0, F_0] \simeq U(Q_0)^m \oplus V$
where $V$ is a finite dimensional submodule and
$H^0(Q_0,  F_0/ [F_0, F_0]) \simeq H^0(Q_0, U(Q_0)^m \oplus V) \simeq $ $ H^0(Q_0, U(Q_0)^m) \oplus H^0(Q_0, V) = H^0(Q_0, V) \subseteq V.$
In particular
$dim_K H^0(Q_0,  F_0/ [F_0, F_0])  \leq dim_K V < \infty.$ Hence $H_2(L_0, K)$ is finite dimensional, so $L_0$ is finitely presented.
\end{proof}

\section{Aplications of Euler characteristic} \label{secEuler}

Let $L$ be a Lie algebra of finite cohomological dimension $cd(L) = n $ and such that for all $i $ and for all finite dimensional $U(L)$-modules $W$ we have $dim_K H_i(L, W) < \infty$. We call such Lie algebras homologically finitary. \iffalse{Note that this condition is satisfied if $L$ is of homological type $FP_{\infty}$.
Indeed since $L$ is of type $FP_{\infty}$ there is a projective resolution
$ \mathcal{P} : 0 \to P_n \to \ldots \to P_1 \to P_0 \to K \to 0$
such that all $P_i$ are finitely generated $U(L)$-modules. Then
$H_i(L, W) = Tor_{i}^{U(L)} (K, W) = H_i (\mathcal{P}^{del} \otimes_{U(L)} W),$
where upper index $del$ means deleted resolution. Note that
\iffalse{$$\mathcal{R} = \mathcal{P}^{del} \otimes_{U(L)} W : 0 \to R_n \to \ldots \to R_1 \to R_0 \to 0,$$
where $R_i = P_i \otimes_{U(L)} W$. Since $P_i$ is finitely generated as a $U(L)$-module and}\fi since $W$ is finite dimensional we conclude that each $P_i \otimes_{U(L)} W$ is finite dimenional. Thus $H_i (\mathcal{P}^{del} \otimes_{U(L)} W)$ is finite dimensional.}\fi
By definition the Euler characteristic of $L$ is defined by
$$
\chi(L) = \sum_{0 \leq i \leq cd(L)} (-1)^i \dim_K H_i(L, K).$$

\begin{lemma} Let $0 \to L_1 \to L \to L_2 \to 0$ be  a short exact sequence of Lie algebras, where $L_1$ and $L_2$ are homologically finitary. Then $L$ is homologically finitary.
\end{lemma}

\begin{proof} Let $W$ be a finite dimensional $U(L)$-module.
Consider the LHS spectral sequence
$E^2_{p,q} = H_p(L_2, H_q(L_1, W))$ that converges to $H_{p+q} (L, W)$.
Since $L_1$ is homologically finitary we have that $ dim_K  H_q(L_1, W) < \infty$. Then using that $L_2$ is homologically finitary we conclude that 
$dim_K E_{p,q}^2 < \infty$. Since $E^{\infty}_{p,q}$ is a subquotient of $E_{p,q}^2$ we conclude that $E_{p,q}^{\infty}$ is finite dimensional, hence each $H_{p+q}(L, W)$ is finite dimensional.
\end{proof}

\begin{theorem} Let $0 \to L_1 \to L \to L_2 \to 0$ be  a short exact sequence of Lie algebras, $L_1$ is homologically finitary  and the trivial $U(L_2)$-module $K$ has a free resolution of finite length with each free module finitely generated.
Then
$\chi(L) = \chi(L_1) \chi(L_2)$.
\end{theorem}

\begin{proof}
Consider the LHS spectral sequence
$E^2_{p,q} = H_p(L_2, H_q(L_1, K))$ that converges to $H_{p+q} (L, K)$.
It is a well known fact that if $\mathcal{C} : \ \ \  \ldots \to C_i \to C_{i-1} \to \ldots$ is a complex of finite length  of finite dimensional ( over $K$) vector spaces then by defining $\chi(\mathcal{C}) = \sum_i (-1)^i dim_K H_i(\mathcal{C})$ we have
$\chi(\mathcal{C}) = \chi(H_*(\mathcal{C})) $ 
where $H_*(\mathcal{C})$ is a complex with zero differentials.
This together with the fact that $E^{r+1} _{*,*}= H_{*,*}(E^r)$ and the fact the spectral sequence is converging implies that
$$\chi(L) =   \sum_{0 \leq i \leq cd(L)} (-1)^i dim_K H_i(L, K)
 = \sum_{p, q \geq 0} (-1)^{p+q} dim_K E^{\infty}_{p,q} =  $$ $$\sum_{p, q \geq 0} (-1)^{p+q} dim_K E^{2}_{p,q} =  \sum_{p, q \geq 0} (-1)^{p+q} dim_K   H_p(L_2, H_q(L_1, K))$$
Let
$$\mathcal{P} :  0 \to U(L_2)^{s_m} \to \ldots  \to U(L_2)^{s_1} \to U(L_2)^{s_0} \to K \to 0$$
be a free resolution of the trivial $U(L_2)$-module $K$ where each module is finitely generated. \iffalse{Then
$$\mathcal{P}^{del} \otimes_{U(L_2)} V_q : 0 \to V_q^{s_m} \to \ldots  \to V_q^{s_1} \to V_q^{s_0} \to 0$$
where}\fi Set
$V_q = H_q(L_1, K).$
Then
$$
\sum_{p} (-1)^{p} dim_K   H_p(L_2, V_q) = \sum_{p} (-1)^{p} dim_K H_p(\mathcal{P}^{del} \otimes_{U(L_2)} V_q ) =$$ $$ \sum_{p} (-1)^{p} dim_K (U(L_2)^{s_p} \otimes_{U(L_2)} V_q ) =\sum_{p} (-1)^{p} s_p dim_K V_q = \chi (L_2) dim_K V_q, \hbox{ hence }
$$
$$
\chi(L) =   
\sum_{q \geq 0} (-1)^{q} \sum_{p \geq 0} (-1)^{p} dim_K   H_p(L_2, V_q) =\sum_{q \geq 0}  \chi (L_2) (-1)^{q} dim_K V_q =  \chi(L_2) \chi(L_1)
$$
\end{proof}

\begin{lemma} \label{Euler0} Suppose $0 \to F_1 \to L \to F_2 \to 0$ be  short exact sequence of Lie algebras with $F_1$ and $F_2$ finitely generated, free, non-abelian Lie algebras. Let $L_0$ be an ideal of $L$ such that $dim_K L/ L_0 = 1$. Then $L_0$ is not finitely presented. 
\end{lemma}
\begin{proof} 
Indeed using cohomological dimensions
 $cd(L_0) \leq cd(L) \leq cd(F_1) + cd(F_2) = 2$. Assume that $L_0$ is finitely presented . Then $\chi(L_0)$ is well-defined and by the above result
 $0 \not= ( 1 - d(F_1)) ( 1 - d(F_2)) = \chi(F_1) \chi(F_2) = \chi(L) = \chi(L_0) \chi(L/ L_0) = \chi(L_0) .0 = 0$
 a contradiction.  
\end{proof}

 \iffalse{

\begin{lemma} \cite{K-MP3} An $\mathbb{N}$-graded Lie algebra $L$ is free if and only if $H_2(L, K) = 0$.
\end{lemma}

\begin{lemma} \label{inf-surf} Let $S$ be a surface Lie subalgebra with $I$  an  ideal of $S$ such that $dim_K S/ I = 1$. Then $I$ is free and if $S$ non-abelian then $I$ is
not finitely generated.
\end{lemma}

\begin{proof} Note that $S$ has $\mathbb{N}$-grading with the standard generators having degree 1. Under this grading $I$ is a graded ideal of $S$.

Let $Q = S/ I$. Consider the spectral sequence $$E^2_{i,j} = H_i(Q, H_j(I, K))$$ converging to $H_{i+j}(S, K)$. Since $cd(Q) = 1 $ we have that $E^2_{i,j} = 0$ for $i \geq 2$, hence the spectral sequence colapces, so $E^{\infty}_{i,j} = E^2_{i,j}$. Hence
$$1 = dim_K H_2(S, K) = dim_K E^{\infty}_{0,2} + dim_K  E^{\infty}_{1,1} =
dim_K E^{2}_{0,2} + dim_K  E^{2}_{1,1} =$$ $$dim_K H_0(Q, H_2(I, K)) + dim_K H_1(Q, H_1(I,K)) $$
If $H_1(Q, H_1(I, K)) \not= 0$ then 
$dim_K H_1(Q, H_1(I, K)) \geq 1$, so $dim_K H_0(Q, H_2(I, K)) = 0$. Since $I$ is a graded subalgebra $H_2(I,K) = 0$ is equivalent to $H_0(Q, H_2(I, K)) = 0$.
Then we apply the previous lemma.

Now using that $I$ is free, suppose that $I$ is finitely generated. Then $\chi(I)$ is well defined and
$$2 - d(S) = \chi(S) = \chi(I) \chi (S/ I) = \chi(I) . 0 = 0$$
a contradiction.

 \end{proof}
 
 }\fi

\end{document}